\documentclass{birkmult}
\usepackage{amsmath, amsthm, amsfonts, amssymb, mathrsfs}
\usepackage{graphicx,xcolor}

\usepackage{mathtools}
\usepackage{enumerate}
\usepackage{enumitem}
\usepackage{marginnote}
\usepackage{caption}
\usepackage{subcaption}
\newtheorem{theorem}{Theorem}[section]
\newtheorem{lemma}[theorem]{Lemma}
\newtheorem{proposition}[theorem]{Proposition}

\theoremstyle{remark}
\newtheorem{remark}[theorem]{Remark}
\theoremstyle{definition}

\newtheorem{ass}{Assumption}

\newcommand{\cH}{{\mathcal{H}}}
\def\cA{{\mathcal A}}

\def\CC{{\mathbb C}}
\def\RR{{\mathbb R}}

\def\NN{{\mathbb N}}

\def\la{{\lambda}}
\def\d{\,{\rm{d}}}
\newcommand\sess{\sigma_{\rm ess}}

\newcommand{\supp}{\operatorname{supp}} 
\newcommand\ds{\displaystyle}

\newcommand{\mydot}{\, \cdot \,}
\DeclareMathOperator*{\esssup}{ess\,sup}
\DeclareMathOperator*{\essinf}{ess\,inf}
\DeclareMathOperator*{\dist}{dist}
\DeclareMathOperator*{\ran}{ran}
\DeclareMathOperator*{\essran}{ess\,ran}
\begin{document}
\author{Orif \,O.\ Ibrogimov}
\address
	{
	Department of Mathematics, 
	University College London,
	Gower Street, London,
	WC1E 6BT, UK}
\email{o.ibrogimov@ucl.ac.uk}
\author{Christiane Tretter}
\address
	{
	Mathematisches Institut, 
	Universit\"{a}t Bern,
	Sidlerstrasse.\!~5,
	3012 Bern, Switzerland}
\email{tretter@math.unibe.ch}
\title[Spectrum of an operator matrix in truncated Fock space]
{On the spectrum of an operator \\ in truncated Fock space}
\dedicatory{To Prof.\,Dr.\,Dr.h.c.\,mult.\ Heinz Langer with admiration}
\thanks{The authors gratefully acknowledge the support of the \emph{Swiss National Science Foundation,} SNF, grant no.\ $200020\_146477$; the first author also gratefully acknowledges the support of SNF Early Postdoc.Mobility grant no.\ $168723$ and thanks the Department of Mathematics at University College London for the kind hospitality.}
\date{\today}
\begin{abstract}
We study the spectrum of an operator matrix arising in the spectral analysis of the energy operator of the spin-boson model of radioactive decay with two bosons on the torus. An analytic description of the essential spectrum is established. 
Further, a criterion for the finiteness of the number of eigenvalues below the bottom of the essential spectrum is derived.
\end{abstract}

\maketitle

\section{Introduction}

In this paper we study the essential spectrum and discrete spectrum of the tridiagonal operator matrix 
\begin{equation}\label{bom:H}
H:=
\begin{pmatrix}
H_{00} & H_{01} & 0\\[1ex]
H_{01}^* & H_{11} & H_{12}\\[1ex]
0 & H_{12}^* & H_{22}\\[1ex]
\end{pmatrix}
\end{equation}
in the so-called truncated Fock space 
\[
\cH:=\cH_0\oplus\cH_1\oplus\cH_2
\]
with $\cH_0:=\CC$, $\cH_1:=L^2(\Omega,\CC)$ and $\cH_2:=L^2_{\rm sym}(\Omega^2,\CC)$. 
Here $\Omega$ is a $d$-dimensional open cube $(-a,a)^d$, $d\in\NN$, $a\in(0,\infty)$, and $L^2_{\rm sym}(\Omega^2,\CC)$ stands for the subspace of $L^2(\Omega^2,\CC)$ consisting of symmetric functions (with respect to the two variables).
The operator entries $H_{ij}: \cH_j \to \cH_i$, $|i-j|\leq1$, $i,j=0,1,2$, are given \vspace{-0.5mm} by 
\begin{align}
\nonumber
&H_{00} f_0 = w_0 f_0, \quad  H_{01}f_1 = \int\limits_{\Omega} v_0(s)f_1(s) \d s,\\[-1mm] \label{def.op.entries}
&(H_{11} f_1)(x) = w_1(x)f_1(x), \quad (H_{12}f_2)(x)=\int\limits_{\Omega}v_1(x,s)f_2(x,s) \d s,\\[1mm] \nonumber
&(H_{22}f_2)(x,y) = w_2(x,y)f_2(x,y), \\[-1.5mm] \nonumber
\end{align}
for almost all (a.a.) $x,y\in\Omega$ with parameter functions satisfying certain rather weak conditions to be specified below.

Operator matrices of this form play a key role for the study of the energy operator of the spin-boson Hamiltonian with two bosons on the torus. In fact, the latter is a $6\times6$ operator matrix which is unitarily equivalent to a $2\times 2$ block diagonal operator with two copies of a particular case of $H$ on the diagonal, see e.g.\ \cite[Section~III]{MNR-2016-1D}. Consequently, the essential spectrum and finiteness of discrete eigenvalues of the spin-boson Hamiltonian are determined by the corresponding 
spectral information on the operator matrix $H$ in \eqref{bom:H}.

Independently of whether the underlying domain is a torus or the whole space~$\RR^d$, the full spin-boson Hamiltonian is an 
infinite operator matrix in Fock space for which rigorous results are very hard to obtain. One line of attack is to consider 
the compression to the truncated Fock space with a finite number $N$ of bosons, 
and in fact most of the existing literature concentrates on the case $N\le 2$. For the case of $\RR^d$ there are some exceptions, 
see e.g.\ H\"ubner, Spohn \cite{Huebner94atominteracting}, \cite{Huebner-Spohn-1995} for arbitrary finite $N$ and Zhukov, Minlos \cite{Zhukov-Milnos-1995} for $N=3$, where a rigorous scattering theory was developed for small coupling constants. 

For the case when the underlying domain is a torus,
the spectral properties of a slightly simpler version of $H$ were investigated by Muminov, Neidhardt and Rasulov \cite{MNR-2016-1D}, Albeverio, Lakaev and Rasulov \cite{Al-La-Ra-2007}, Lakaev and Rasulov \cite{La-Ra-2003}, Rasulov~\cite{Rasulov-2016TMF}, see also the references therein. In the case when $v_1$ is a function of a single variable and all parameter functions are continuous (sometimes even real-analytic) with special properties on a closed torus of specific dimension, an analytic description of the essential spectrum was first given in \cite{La-Ra-2003}; a Birman-Schwinger type result was 
first established in \cite{Al-La-Ra-2007}; the finiteness of the discrete spectrum was analysed in \cite{MNR-2016-1D} for $d=1$ with real-analytic 
parameter functions. 

In this paper we establish an analytic description of the essential spectrum, a Birman-Schwinger type result as well as a criterion 
guaranteeing the finiteness of discrete eigenvalues below the bottom of the essential spectrum of $H$.
Compared to earlier work, we achieve these results in a more general setting with \pagebreak weaker conditions on the parameter functions.
For example, the dimension $d\in\NN$ is arbitrary, the parameter function $v_1$ is required to be neither of one variable nor
real-analytic or continuous. In fact, our analysis shows that it suffices to require the boundedness of the functions
$x\mapsto \|v_1(x,\cdot)\|_{L^{2+\varepsilon}(\Omega)}$ and $y\mapsto\|v_1(\cdot,y)\|_{L^{2+4/\varepsilon}(\Omega)}$ on $\Omega$ for some $\varepsilon>0$. Although we consider the case $a<\infty$ throughout the paper, our methods are of local nature 
and also apply to the case $a=\infty$ where $\Omega=\RR^d$.
<
Another difference to earlier work is that we employ more abstract me-thods, allowing for simpler proofs of the first two results mentioned above; in particular, we do not make use of the so-called generalized Friedrichs model in our analysis. However, in spite of being self-adjoint and bounded (with compact underlying domain), the operator matrix $H$ in \eqref{bom:H} is, up to our knowledge, not covered by any of the currently existing abstract results such as \cite{ALMS94}, \cite{AMS-1998}, \cite{MR2363469}, \cite{Kraus-Langer-Tretter-2004}, \cite{Langer-Langer-Tretter-2002}. 

The abstract results on the essential/discrete spectrum in \cite{ALMS94}, \cite{AMS-1998}, \cite{MR2363469} do not apply since the required compactness assumptions on certain auxiliary operators are violated mainly due to the non-compactness of partial-integral operators. The variational principles of \cite{Kraus-Langer-Tretter-2004}, \cite{Langer-Langer-Tretter-2002} do not give information on the finiteness/infiniteness of discrete eigenvalues either because none of the diagonal entries of $H$ has infinitely many discrete eigenvalues. For the present approach, since the last diagonal entry $H_{22}$ of $H$ is a multiplication operator, it turned out to be natural to use singular sequences to describe one part of the essential spectrum and to employ a Schur complement approach to describe the second~part. 

We mention that, in a more concrete setting, the infiniteness of the discrete eigenvalues below the bottom of the essential spectrum of $H$ and corresponding eigenvalue 
asymptotics were also discussed in the literature, see e.g.\ Albeverio, Lakaev and Rasulov \cite{Al-La-Ra-2007}; these results were obtained using the machinery developed in Sobolev \cite{Sobolev-CMP93}. 
To achieve analogous results in~our general setting seems to be very challenging and is beyond the scope of this paper.

The paper is organized as follows. In Section~2 we formulate the hypotheses on the parameter functions, explain the reduction of the problem to a $2\times2$ operator 
matrix and describe the Schur complement of the latter. In Sections~3 and 4 we establish the analytic description of the essential 
spectrum and a Birman-Schwinger type result, respectively. In Section~5, inspired by the methods of \cite{Al-La-Ra-2007}, we derive the criterion for the finiteness of the discrete spectrum below the bottom 
of the essential spectrum of $H$. Section~6 contains some concluding remarks e.g.\ on the limiting case $a=\infty$ and on 
modifications of the assumptions under which our results continue to hold.

The following notations will be used in the sequel: ${\rm{cl}}\,(X)$ denotes the closure of a set $X\subset\RR^d$ in $\RR^d$ (w.r.t. the standard topology); 
for a complex-valued function $\varphi$, we denote by $\varphi^*$ the complex conjugate of $\varphi$; 
$\ran(f)$ and $\essran(f)$ respectively denote the range and the essential range of a (measurable) function $f$ on $\Omega$ or $\Omega^2$, respectively; 
a function $f$ on $\Omega^2$ is called symmetric if $f(x,y) = f^*(y,x)$ for (a.a.\ if applicable) $x,y \in \Omega$.

\pagebreak

\section{The block operator matrix}
Throughout the paper we assume that the parameter functions in \eqref{def.op.entries} satisfy the following hypotheses.
\begin{ass}\label{ass:A}
$w_0\in\RR, v_0\in L^2(\Omega, \CC)$, $w_1\in L^\infty(\Omega, \RR)$, $w_2\in C(\Omega^2)\cap L^\infty(\Omega^2, \RR)$ with $w_2(x,y)=w_2(y,x)$, $x,y\in\Omega$. For some ~$\varepsilon>0$ the functions $x\mapsto v_1(x,\mydot)$ and $y\mapsto v_1(\mydot,y)$ belong to $L^\infty(\Omega, L^{2+\varepsilon}(\Omega,\CC))$ 
and $L^\infty(\Omega, L^{2+4/\varepsilon}(\Omega,\CC))$, respectively, i.e.
\begin{align}\label{assA:ess.sup.cond}
 \esssup_{x\in\Omega}\|v_1(x,\mydot)\|_{L^{2+\varepsilon}(\Omega)}<\infty, \quad  \esssup_{y\in\Omega}\|v_1(\mydot,y)\|_{L^{2+4/\varepsilon}(\Omega)}<\infty.
\end{align}
\end{ass}

\begin{remark}\label{rem:reduction}
(i) Under Assumption~\ref{ass:A}, it is easy to see that $H:\cH\to\cH$ is an everywhere defined bounded self-adjoint operator. 

\smallskip
\noindent
(ii) Since $H_{00}$, $H_{01}$ and $H_{10}$ are finite-rank operators and the essential spectrum as well as the finiteness of (parts of) the discrete spectrum of 
self-adjoint operators are invariant with respect to finite-rank perturbations (see e.g. \cite[Chapter~IX]{Birman-Solomjak-87b}), we can restrict ourselves to studying 
the spectrum of the $2\times 2$ operator matrix 
\begin{equation}
\cA:=
\begin{pmatrix}
H_{11} & H_{12}\\[1.5ex]
H^*_{12} & \!H_{22}
\end{pmatrix}.
\end{equation}
acting in the Hilbert space $\cH_1\oplus\cH_2$. 

\smallskip
\noindent
(iii) Since ${\rm{Vol}}\,(\Omega)<\infty$, H\"older's inequality together with the second condition in \eqref{assA:ess.sup.cond} yields that $v_1(x,\mydot)\in L^2(\Omega, \CC)$ 
for a.a.\ $x\in\Omega$ and
\begin{align}\label{rem3}
\ds\esssup_{x\in\Omega}\|v_1(x,\mydot)\|_{L^2(\Omega)} \leq (2a)^{\frac{\varepsilon d}{4+2\varepsilon}}  
\esssup_{x\in\Omega}\|v_1(x,\mydot)\|_{L^{2+\varepsilon}(\Omega)}<\infty.
\end{align}
\end{remark}

It is easy to check that the adjoint operator $H^*_{12}:\cH_1\to\cH_2$ is given by
\[
 (H^*_{12}f)(x,y) = \frac{1}{2}v_1(x,y)^*f(x)+\frac{1}{2}v_1(y,x)^*f(y), \quad f\in\cH_1,
\]
for a.a.\ $x,y\in\Omega$.

Schur complements have proven to be useful tools when dealing with $2\times 2$ operator matrices (see e.g.\ \cite{Tre08}). The first Schur complement associated with 
the operator matrix $\cA-z$ is given by 
\begin{align*}
 S(z) &= H_{11}-z-H_{12}(H_{22}-z)^{-1}H^*_{12} =: \Delta(z)+K(z)
\end{align*}
for $z\notin\sigma(H_{22})={\rm{cl}}\,(\ran w_2)$ where $\Delta(z):\cH_1\to\cH_1$ is the multiplication operator by the function $\Delta(\mydot;z)$ 
defined as
\begin{equation}\label{def:Delta}
\Delta(x;z) := w_1(x)-z-\frac{1}{2}\int_{\Omega}\frac{|v_1(x,y)|^2}{w_2(x,y)-z} \d y, \quad x\in\Omega,
\end{equation}
and $K(z)\!: \cH_1\to\cH_1$ is the integral operator with kernel $K(\mydot,\mydot;z)$ given by  
\begin{equation}\label{kernel:K}
K(x,y;z) := -\frac{1}{2}\frac{v_1(x,y) v_1(y,x)^*}{w_2(x,y)-z}, \quad (x,y)\in\Omega^2.
\end{equation}

For every $z\!\in\!\RR\setminus {\rm{cl}}\,(\ran w_2)$, the Schur complement $S(z)$ is bounded and self-adjoint in $\cH_1$, the function $\Delta(\mydot; z)$ is real-valued 
and $K(x,y;z)=K(y,x;z)^*$, $x,y\in\Omega$; thus the operators $\Delta(z)$ and $K(z)$ are self-adjoint, too. Moreover, it follows from \eqref{rem3} that 
$\esssup_{x\in\Omega}\Delta(x;z)<\infty$ for every $z\in\RR\setminus{\rm{cl}}\,(\ran w_2)$. Therefore, the multiplication operator $\Delta(z):\cH_1\to\cH_1$ 
is bounded for every $z\in\RR\setminus{\rm{cl}}\,(\ran w_2)$, and hence so is $K(z):\cH_1\to\cH_1$. In fact, we have more than just the boundedness of the integral 
operator $K(z)$ as a corollary of the next lemma.

\begin{lemma}\label{K=Hilber-Schmidt}
Let Assumption~\ref{ass:A} be satisfied. For every $z\in\RR\setminus{\rm{cl}}\,(\ran w_2)$, the integral operator $K(z):\cH_1\to\cH_1$ is Hilbert-Schmidt. 
\end{lemma}
\begin{proof}
Let $z\in\RR\setminus{\rm{cl}}\,(\ran w_2)$ be fixed. By Young's inequality, we have  
\begin{align*}
\ds |v_1(x,y)|^2|v_1(y,x)|^2 \leq \frac{2}{2+\varepsilon} |v_1(x,y)|^{2+\varepsilon} + \frac{\varepsilon}{2+\varepsilon}|v_1(y,x)|^{2+4/\varepsilon}
\end{align*}
for a.a.\ $x,y\in\Omega$. Therefore, 
\begin{align*}
 |K(x,y;z)|^2 \leq \frac{2|v_1(x,y)|^{2+\varepsilon} + \varepsilon|v_1(y,x)|^{2+4/\varepsilon}}{4(2+\varepsilon)\dist(z,\ran w_2)^2} 
\end{align*}
for a.a.\ $x,y\in\Omega$. On the other hand, in view of Assumption~\ref{ass:A}, it is easy to see that the following estimates hold
\begin{align*}
& \ds\|v_1\|_{L^{2+\varepsilon}(\Omega^2)}^{2+\varepsilon} \leq (2a)^d \esssup_{x\in\Omega}\|v_1(x,\mydot)\|_{L^{2+\varepsilon}(\Omega)}^{2+\varepsilon}<\infty,\\
& \ds  \ds\|v_1\|_{L^{2+4/\varepsilon}(\Omega^2)}^{2+4/\varepsilon} \leq (2a)^d \esssup_{y\in\Omega}\|v_1(y,\mydot)\|_{L^{2+4/\varepsilon}(\Omega)}^{2+4/\varepsilon}<\infty.
\end{align*}
Hence $K(\mydot,\mydot;z)\in L^2(\Omega^2)$ and thus $K(z)$ is Hilbert-Schmidt.
\end{proof}

\section{Analytic description of the essential spectrum}
The following theorem provides an explicit formula for the essential spectrum of $H$ in terms of the functions $w_2$ and $\Delta$ given by \eqref{def.op.entries} 
and \eqref{def:Delta}. 
\begin{theorem}\label{thm:ess.spec.1}
Let Assumption~\ref{ass:A} be satisfied and let 
\begin{align}\label{m,M}
m:=\!\inf_{(x,y)\in\Omega^2} \! w_2(x,y), \quad M:=\!\sup_{(x,y)\in\Omega^2} \! w_2(x,y).
\end{align}
Then 
\begin{equation}
\sess(H)=\Sigma_1\cup\Sigma_2 
\end{equation}
where 
\begin{equation*}
\Sigma_1 := {\rm{cl}}\,(\ran w_2)=[m,M], \quad \Sigma_2 := {\rm{cl}}\,\{z\in\RR\setminus\Sigma_1: \; 0\in\essran\Delta(\mydot;z)\}.
\end{equation*}
\end{theorem}
\begin{proof}
Recall that, by Remark~\ref{rem:reduction}, $\sess(H)=\sess(\cA)$ and thus it suffices to establish that $\sess(\cA)=\Sigma_1\cup\Sigma_2$.
First we show $\Sigma_1 \subset \sess(\cA)$. Since the essential spectrum is closed, we only have to prove the inclusion 
\begin{align}\label{inclusion1}
\{w_2(x,y) :  x,y\in \Omega\} \subset  \sess(\cA).
\end{align}
To this end, let $z_0 \in \{w_2(x,y) :  x,y\in \Omega\}$ be arbitrary. Since $w_2:\Omega^2\to\RR$ is continuous on $\Omega^2$ by Assumption~\ref{ass:A}, it follows 
that $z_0=w_2(x_0,y_0)$ for some $(x_0, y_0) \in \Omega^2$. 

Let $\chi$ be the normalized characteristic function of the annulus 
$\bigl\{x\!\in\! \Omega : \!\frac{1}{2} \leq \|x\| \leq 1\bigr\}$ and define the sequences 
$\{\varphi_n\}_{n \in \NN}$, $\{\phi_n\}_{n \in \NN} \subset\cH_1=L^2(\Omega,\CC)$ by
\begin{align*}
\varphi_n(x) := 2^{\frac{nd}{2}} \chi(2^n(x-x_0)), \quad \phi_n(y) := 2^{\frac{nd}{2}} \chi(2^n(y-y_0)), \quad x,y \in \Omega;
\end{align*}
note that $\varphi_n=\phi_n$ if $x_0=y_0$. 
It is easy to check that
\[
 \supp(\varphi_n) \cap \supp(\varphi_m) = \supp(\phi_n) \cap \supp(\phi_m)=\emptyset
\]
for all $n,m\in\NN$ with $n\neq m$ and that there is $N_0\in\NN$ such that
\[
\|\varphi_n\|_{L^2(\Omega)}=\|\phi_n\|_{L^2(\Omega)}=1, \quad \supp(\varphi_n) \cap \supp(\phi_k)= \emptyset
\] 
for all positive integers $n,k\geq N_0$. So both $\{\varphi_n\}^\infty_{n=N_0}$ and $\{\phi_n\}^\infty_{n=N_0}$ are orthonormal systems in $\cH_1$.

Now consider the sequence $\{\psi_n\}^\infty_{n=N_0}$ defined by 
\begin{equation}
\psi_n(x,y)=\begin{cases}
\varphi_n(x) \phi_n(y) = \varphi_n(x) \varphi_n(y) &\mbox{if } x_0=y_0,\\
\frac{1}{\sqrt{2}}\bigl(\varphi_n(x)\phi_n(y)+\varphi_n(y)\phi_n(x)\bigr) &\mbox{if } x_0\neq y_0,
\end{cases}
\end{equation}
for $x,y\in\Omega$. It is easy to see that the sequence $\{\psi_n\}^\infty_{n=N_0}$ is an orthonormal system in $\cH_2=L^2_{\rm{sym}}(\Omega,\CC)$. Hence the sequence $\{\widetilde{\psi}_n\}^\infty_{n=N_0}$ given by 
\begin{align*}
\widetilde\psi_n(x,y)= \binom{0}{\psi_n(x,y)}, \quad x,y \in \Omega,
\end{align*}
is an orthonormal system in $\cH_1\oplus\cH_2$. Thus, if we show $\|(\cA-z_0) \widetilde\psi_n\|_{\cH} \to 0$ 
as $n\to\infty$, it follows that $\{\widetilde{\psi}_n\}^\infty_{n=N_0}$ is a singular sequence for $\cA-z_0$ and thus $z_0\in\sess(\cA)$, see ~\cite{Reed-Simon-IV}. Note that 
\begin{align*}
\|(\cA-z_0) \widetilde\psi_n\|^2_{\cH}= \|H_{12} \psi_n\|^2_{L^2(\Omega)}+ \|(H_{22}-z_0) \psi_n\|^2_{L^2(\Omega^2)}.
\end{align*}
By construction of the sequence $\{\psi_n\}_{n \in \NN}$, it easily follows that 
\[
 \|(H_{22}-z_0) \psi_n\|_{L^2(\Omega^2)} = \|(w_2-w_2(x_0,y_0))\psi_n\|_{L^2(\Omega^2)} \to 0, \quad n \to \infty,
 \]
so it is left to be shown that 
\begin{equation}\label{sing.seq.H_12}
 \|H_{12} \psi_n\|_{L^2(\Omega)} \to 0, \quad n \to \infty.
\end{equation}
By Assumption~\ref{ass:A}, there are constants $C>0$ and $\varepsilon>0$ such that 
\[
\|v_1(x, \mydot) \|_{L^{2+\varepsilon}(\Omega)} \leq C
\]
for a.a.\ $x \in \Omega$.
Applying H\"older's inequality with $p=2+\varepsilon$ and $q= \frac{p}{p-1}$, we thus obtain 
\begin{align*}
\begin{split}
\left|\int_{\Omega} v_1(x,y) \varphi_n(y) \d y\right| &\leq \| v_1(x, \mydot)\|_{L^p(\Omega)} \|\varphi_n\|_{L^q(\Omega)} 
 \leq C 2^{n d(\frac{1}{2}-\frac{1}{q})}
\end{split}
\end{align*}
for a.a.\ $x \in \Omega$. In the same way it follows that
\begin{align*}
\begin{split}
\left|\int_{\Omega} v_1(x,y) \phi_n(y) \d y\right| \leq C 2^{n d(\frac{1}{2}-\frac{1}{q})} 
\end{split}
\end{align*}	
for a.a.\ $x \in \Omega$.
Therefore, using $\|\varphi_n\|_{L^2(\Omega)}=\|\phi_n\|_{L^2(\Omega)}=1$ and applying the triangle inequality, we easily obtain
\begin{equation}\label{new estimate}
\|H_{12} \psi_n\|_{L^2(\Omega)} \leq C 2^{n d(\frac{1}{2}-\frac{1}{q})+1}, \quad n\geq N_0. 
\end{equation}
This proves \eqref{sing.seq.H_12} because $q= 2-\frac{\varepsilon}{1+\varepsilon}<2$.
%

\vspace{1mm}

Now it remains to be shown that $(\RR\setminus\Sigma_1) \cap \sess(\cA) = \Sigma_2$. To this end, let $z \in \RR \backslash \Sigma_1$ be arbitrary. It is not difficult to 
check that \cite[Theorem~2.4.7]{Tre08} applies and yields 
\begin{equation}\label{ess.equiv1}
z \in \sess(\cA) \quad \Longleftrightarrow \quad 0 \in \sess(S(z)).
\end{equation}
Since $K(z)$ is compact, 
we have $\sess(S(z)) \!=\! \sess(\Delta(z))\!=\!\essran(\Delta(\mydot;z))$. Therefore, by \eqref{ess.equiv1},
\begin{equation*}
z \in \sess(\cA) \quad \Longleftrightarrow \quad 0\in\essran \Delta(\mydot;z) \quad \Longleftrightarrow \quad z \in \Sigma_2.\qedhere
\end{equation*}
\end{proof}

\vspace{2mm}

\begin{remark}
While it is always the case that $\Sigma_1\neq\emptyset$, the following example shows that $\Sigma_2=\emptyset$ may occur. 

Let $d\in\NN$ be arbitrary and let $\Omega=(-a,a)^d$ with $a=2^{(1-d)/d}$ so that ${\rm{vol}}\,(\Omega)=2$. Let $w_2$ be an arbitrary function satisfying Assumption~\ref{ass:A} 
and denote its continuous extension to ${\rm{cl}}\,(\Omega^2)$ also by $w_2$. If $m,M$ are defined as in \eqref{m,M} and we choose the parameter functions $w_1$ and $v_1$ as 
\begin{align*}
 & v_1(x,y) = (w_2(x,y)-m)^{1/2}(M-w_2(x,y))^{1/2}, \quad x,y\in {\rm{cl}}\,(\Omega), \\
 & w_1(x) = m+M-\frac{1}{2}\int_\Omega w_2(x,y) \d y, \quad x\in {\rm{cl}}\,(\Omega),
\end{align*}
then, clearly, 
Assumption~\ref{ass:A} is satisfied and $\Delta(\mydot;m) \equiv \Delta(\mydot;M) \equiv 0$ on ${\rm{cl}}\,(\Omega)$. On the other hand, it is easy to see that 
the function $z\mapsto\Delta(x;z)$ is strictly 
decreasing on $(-\infty,m) \cup (M,\infty)$ for each fixed $x\in {\rm{cl}}\,(\Omega)$. Therefore, for each $z<m$, we have 
$\Delta(x;z)>\Delta(x;m)=0$ for all $x\in {\rm{cl}}\,(\Omega)$ and for each $z>M$, we have $\Delta(x;z)<\Delta(x;M)=0$ for all $x\in {\rm{cl}}\,(\Omega)$. 
Consequently, $\Sigma_2=\emptyset$.
\end{remark}

\section{Birman-Schwinger type principle}
For a bounded self-adjoint operator $A:\cH\to\cH$ and a constant $\la \in \RR$, we define the quantity
\begin{align*}
n(\la ; A) := \sup_{\mathcal{L} \subset \cH} \{ \dim \mathcal{L} : (Au,u) >\la \|u\|_{\cH}^2, u \in \mathcal{L}\} = \dim \mathcal{L}_{(\la,\infty)}(A),
\end{align*}
where $\mathcal{L}_{(\la,\infty)}(A)$ is the spectral subspace of $A$ corresponding to the interval $(\la,\infty)$. Note that, if $n(\la; A)$ is finite, then it is equal to 
the number of the eigenvalues of $A$ larger than $\la$ (counted with multiplicities), see e.g., \cite[Section~IX]{Birman-Solomjak-87b}.
For $\la \leq \min \sess(A)$, we denote by $N(\la;A)$ the number of eigenvalues of $A$ that are less than $\la$; observe that, for $ z < \min \sess(A)$,
\begin{align}
N(z; A)= n(-z; -A).
\end{align}
In the sequel, we will use the so-called Weyl inequality (see e.g.\ \cite{Birman-Solomjak-87b})
\begin{equation}\label{Weyl_iequality}
 n(\la_1+\la_2; V_1+V_2) \leq n(\la_1;V_1)+n(\la_2;V_2)
\end{equation}
for compact self-adjoint operators $V_1$, $V_2:\cH\to\cH$, and real numbers $\la_1,\la_2$.

The following lemma plays a crucial role in the analysis of the discrete spectrum.

\begin{lemma}\label{lem:Schur-discrete}
Let Assumption~\ref{ass:A} be satisfied. For every $z<\min\sess(\cA)$,
\begin{align}\label{Birman-Schwinger-Schur}
N(z;\cA) = N(0; S(z)).
\end{align}
\end{lemma}
\begin{proof}
Let $ z < \min \sess(\cA)$ be fixed. Then $z<\min\sigma(H_{22})$ by Theorem~\ref{thm:ess.spec.1}. In the Hilbert space $\cH_1\oplus\cH_2$, we consider the operators 
\begin{align*}
 W(z):={\rm{diag}}(S(z), H_{22}-z), \quad V(z):=\begin{pmatrix}
                                               I & \:-H_{12}(H_{22}-z)^{-1}\\[1ex]
                                               0 & I
                                               \end{pmatrix}.
\end{align*}
Clearly, $V(z):\cH\to\cH$ is bijective and $W(z)=V(z)(\cA-z)V(z)^*$ due to the Frobenius-Schur factorization, see e.g.\ \cite{ALMS94}. Therefore,
\begin{align*}
 N(z;\cA) &= N(0;\cA-z)\\
          &=N(0;V(z)(\cA-z)V(z)^*)\\
          &=N(0;W(z)).
\end{align*}
On the other hand, 
\begin{equation*}
N(0;W(z))=N(0;S(z))+N(0;H_{22}-z)=N(0;S(z)).\qedhere
\end{equation*}
\end{proof}

\vspace{2mm}

\begin{lemma}\label{lem:Delta-positive}
Let Assumption~\ref{ass:A} be satisfied. For every $z < \min \sess(\cA)$, we have 
\[
\ds \essinf_{x \in \Omega} \Delta (x;z) > 0.
\]
\end{lemma}
\begin{proof}
Suppose, to the contrary, that there exists $z^*<\min\sess(\cA)$ such that 
\[
\essinf_{x \in \Omega} \Delta (x;z^*) \leq 0.
\]
Then we must have
\begin{equation}\label{ess.inf.delta}
 \essinf_{x \in \Omega} \Delta (x;z^*) < 0,
\end{equation}
for otherwise we would have $z^*\in\sess(\cA)$ contradicting $z^*<\min\sess(\cA)$.

By \eqref{ess.inf.delta} and Assumption~\ref{ass:A}, there exists a sequence $\{x_n\}_{n\in\NN}\subset \Omega$ satisfying the conditions
\begin{align}\label{x_n}
\Delta(x_n;z^*)<0, \quad  v_1(x_n,\mydot)\in L^{2+\varepsilon}(\Omega),  \quad n\in\NN.
\end{align}
Consider the sequence of functions $z\mapsto \Delta(x_n;z)$, $n\in\NN$, on $(-\infty,z^*]$. For every fixed $n\in\NN$, it is easy to see that 
\begin{equation}\label{lim-Delta=infty}
  \lim_{z\to-\infty}\Delta(x_n;z)=+\infty,
\end{equation}
and $z\mapsto \Delta(x_n;z)$ is a continuous, strictly decreasing function on $(-\infty, z^*]$~as
\begin{equation}
\frac{\partial}{\partial z}\Delta(x_n;z)= -1-\frac{1}{2}\int_{\Omega}\frac{|v_1(x_n,y)|^2}{(w_2(x_n,y)-z)^2} \d y\leq -1.
\end{equation}
Hence, in view of the first condition in \eqref{x_n}, the mean-value theorem implies that there exists a sequence $\{z_n\}_{n\in\NN}\subset (-\infty,z^*)$ 
with $\Delta(x_n;z_n)=0$ for each $n\in\NN$. On the other hand, \eqref{lim-Delta=infty} implies that the sequence $\{z_n\}_{n\in\NN}$ is bounded from below, too. Therefore, 
by Bolzano-Weierstrass' theorem, there is a subsequence $\{z_{n_k}\}_{k\in\NN}$ converging to some $z_0\in(-\infty, z^*]$. If we write, 
\begin{align*}
 \Delta(x_{n_k};z_0) &= \Delta(x_{n_k};z_0) - \Delta(x_{n_k};z_{n_k}) \\[2ex]
                 &= (z_{n_k}-z_0)\Bigg(1+\frac{1}{2}\int_{\Omega}\frac{|v_1(x_{n_k},y)|^2}{(w_2(x_{n_k},y)-z_0)(w_2(x_{n_k},y)-z_{n_k})} \d y\Bigg),
\end{align*}
it follows from the second relation in \eqref{x_n} (see also Remark~\ref{rem:reduction} (iii)) that the integral in the bracket is finite and thus 
$\Delta(x_{n_k};z_0)\to0$, $k\to\infty$. Therefore, Theorem~\ref{thm:ess.spec.1} shows that $z_0\in\sess(\cA)$, contradicting $z_0<\min\sess(\cA)$.
\end{proof}

It follows from Lemma~\ref{lem:Delta-positive} that the function
 \begin{align}
 x \mapsto \Delta(x;z)^{-1/2}, \quad x \in \Omega,
 \end{align}
is well-defined and bounded for every $ z < \min \sess(\cA)$. Let $T(z)$ be the integral operator with kernel 
\begin{align}
T(x,y;z) := - \Delta(x;z)^{-1/2} K(x,y;z) \Delta(y;z)^{-1/2}, \quad (x,y)\in\Omega^2, 
\end{align}
where $K(\mydot, \mydot; z)$ is defined as in \eqref{kernel:K}.

\begin{proposition}\label{prop:Birman-Schwinger}
Let Assumption~\ref{ass:A} be satisfied and let $ z < \min \sess(\cA)$ be arbitrary. Then $T(z) : \cH_1 \to \cH_1$ is Hilbert-Schmidt and
\begin{align}
N(z; \cA)= n(1; T(z)).
\end{align}
\end{proposition}
\begin{proof}
Let $ z < \min \sess(\cA)$ be fixed. Lemma~\ref{lem:Delta-positive} implies that $\Delta(z)^{-1/2}:\cH_1\to\cH_1$ is well-defined and positive operator. 
Since $K(z):\cH_1\to\cH_1$ is Hilbert-Schmidt by Lemma~\ref{K=Hilber-Schmidt}, and $\Delta(z)^{-1/2}:\cH_1\to\cH_1$ is bounded, it follows that 
the operator 
\begin{align*}
T(z)= - \Delta(z)^{-1/2} K(z) \Delta(z)^{-1/2}
\end{align*}
is Hilbert-Schmidt, too. Recalling that $S(z)= \Delta(z)+ K(z)$, ~we have   
\begin{align*}
\begin{split}
\Delta(z)^{-1/2} S(z) \Delta(z)^{-1/2} &=I + \Delta(z)^{-1/2} K(z) \Delta(z)^{-1/2} = I- T(z).
\end{split}
\end{align*}
Therefore,
\begin{align*}
n(1, T(z)) &= N(-1; -T(z)) = N(0; I- T(z)) \\
           &= N\bigl(0; \Delta(z)^{-1/2} S(z) \Delta(z)^{-1/2}\bigr)\\
           &= N(0; S(z)).
\end{align*}
Applying Lemma~\ref{lem:Schur-discrete}, we thus obtain
\begin{equation*}
N(z;\cA) = N(0; S(z)) = n(1, T(z)). \qedhere
\end{equation*}
\end{proof}

\section{Criterion for the finiteness of the discrete spectrum below 
	the bottom of the essential spectrum}
For $\delta>0$, denote by $B_{\delta}(0)$ the ball of radius $\delta$ with centre at the origin in $\RR^d$. For $s\geq0$, we define functions $\Phi_s:\Omega^2\to\RR$,
\begin{align}\label{Phi_s}
\Phi_{s}(x,y) =
\begin{cases}
\|x\|^s+\|y\|^s & {\rm{if}} \quad x\in B_{\delta}(0)\times B_{\delta}(0), \\
\hspace{7mm} 1 & {\rm{otherwise}}.
\end{cases}
\end{align}

\begin{ass}\label{ass:B}
	There exist constants $\alpha \geq 0$, $\beta\in\RR$, $C_1,C_2>0$, ~$\delta\!\in\!(0,a)$ and a unique point $(t_0,t_0)\in\Omega^2$ such that, for a.a.\ $x,y\in\Omega$,
	\begin{align}\label{cond:w_2}
	&w_2(x,y)-\min\sess(H) \geq C_1 \Phi_{\alpha}(x-t_0, y-t_0), \\
	&\chi_{B_{\delta}(t_0)}(y)\, |v_1(x,y)|\leq C_2 \Phi_\beta(0,y-t_0). 
	\end{align}
\end{ass}
We denote by $\alpha^*$ and $\beta^*$, respectively, the infimum and the supremum of the values of $\alpha$ and $\beta$ satisfying Assumption~\ref{ass:B}.

\begin{remark}\label{rem:disc.spec}
	If $\min\sess(H)\notin\Sigma_1$, then $\alpha^*=0$ and the function 
	\begin{equation}\label{Delta(E_min)}
	\Delta(\mydot ; \min\sess(H)):\Omega\to [0,\infty)
	\end{equation}
	is well-defined. If $\min\sess(H)\in\Sigma_1$ and Assumption~\ref{ass:B} is satisfied, then the function \eqref{Delta(E_min)} is well-defined provided that $\alpha^*<2\beta^*+d$.
\end{remark} 

\begin{ass}\label{ass:C}
	There exist constants $\gamma \geq0$ and $C_3>0$ such that for $\delta$ as in Assumption~\ref{ass:B}, whenever $\alpha^* < 2\beta^*+d$, then for 
	a.a.\ $x\in\Omega$,
	\begin{align}
	\Delta(x; \min\sess(H))\geq C_3\Phi_{\gamma}(x-t_0,0).
	\end{align}
\end{ass}
We denote by $\gamma^*$ the infimum of the values of $\gamma$ satisfying Assumption~\ref{ass:C}.

\begin{theorem}\label{thm:fin.disc.spec}
	Let Assumptions~\ref{ass:A}, \ref{ass:B} and \ref{ass:C} be satisfied and let
	\begin{equation}\label{cond:alpha,beta,gamma}
	\alpha^*+\gamma^* < 2\beta^*+d.
	\end{equation}
	Then the operator $H$ has a finite number of eigenvalues {\rm{(}}counted with multiplicities{\rm{)}} below $\min\sess(H)$.
\end{theorem}
\begin{proof}
	Throughout the proof we adopt the notation $E_{\min}:=\min \sess(\cA)$. Recall that, in view of Remark~\ref{rem:reduction}, 
	$\min\sess(H)=\min\sess(\cA)$ and
	\[
	N(E_{\min};H)<\infty \quad \Longleftrightarrow \quad N(E_{\min};\cA)<\infty.
	\]
	
	By Proposition~\ref{prop:Birman-Schwinger}, $T(z)$ is Hilbert-Schmidt for all $z<\min\sess(\cA)$. Next, we show that $T(E_{\min})$ is Hilbert-Schmidt as well.
	
	It follows from Assumptions~\ref{ass:A}, \ref{ass:B} and \ref{ass:C} that the kernel of $T(E_{\min})$ is square-integrable if the function 
	\begin{equation}
	(x,y) \,\mapsto\, \|x\|^{\beta^*-{\gamma^*}/2} \|y\|^{\beta^*-{\gamma^*}/2}\frac{1}{\|x\|^{\alpha^*}+\|y\|^{\alpha^*}}
	\end{equation}
	is square-integrable over $B_{\delta}(0) \times B_{\delta}(0)$. Passing to generalized polar coordinates, it is easy to see that the latter is equivalent to
	\begin{equation}\label{int.conv.1}
	\ds\int_0^\delta \int_0^\delta \frac{1}{(r_1^{\alpha^*}+r_2^{\alpha^*})^2} \, r_1^{2\beta^*-\gamma^*+d-1} r_2^{2\beta^*-\gamma^*+d-1} \d r_1 \d r_2 <\infty. 
	\end{equation}
	Using the elementary inequality between the arithmetic and geometric means,
	\[
	\frac{r_1^{\alpha^*}+r_2^{\alpha^*}}{2} \geq \sqrt{r_1^{\alpha^*} r_2^{\alpha^*}} ,  
	\]
	it is not difficult to check that \eqref{int.conv.1} holds if 
	\begin{equation*}\label{int.conv.2}
	\int_0^\delta t^{(2\beta^*+d)-(\alpha^*+\gamma^*)-1} \d t < \infty,
	\end{equation*} 
	which, in turn, holds if and only if $\alpha^*+\gamma^* < 2\beta^*+d$. Therefore, $T(E_{\min})$ is Hilbert-Schmidt if \eqref{cond:alpha,beta,gamma} is satisfied. 
	
	Summing up, $T(z)$ is Hilbert-Schmidt for every $z\leq E_{\min}$. Moreover, it is an immediate consequence of Lebesgue's dominated convergence theorem that 
	the map $T(\mydot): (-\infty, E_{\min}] \to L(\cH_1)$ is continuous.
	
	Next, let $z\leq E_{\min}$ be arbitrary. Since $T(z)$ is compact, we obviously have 
	\begin{equation}
	n(1/2; T(z))<\infty.
	\end{equation}
	Using the Weyl inequality \eqref{Weyl_iequality} for the compact self-adjoint operators 
	\[
	V_1=T(E_{\min}), \quad V_2=T(z)-T(E_{\min}),
	\]
	and $\la_1=\la_2=1/2$, we obtain
	\[
	n(1;T(z)) \leq n(1/2; T(E_{\min})) + n(1/2; T(z)-T(E_{\min})).
	\]
	Since $T(\mydot): (-\infty, E_{\min}] \to L(\cH_1)$ is continuous, we thus have
	\begin{align*}
	\lim_{z\nearrow E_{\min}} \!\! n(1;T(z)) & \leq n(1/2; T(E_{\min})) + \lim_{z\nearrow E_{\min}} \!\! n(1/2; T(z)-T(E_{\min}))\\
	&= n(1/2; T(E_{\min})).
	\end{align*}
	This together with Proposition~\ref{prop:Birman-Schwinger} yields
	\begin{equation*}
	N(E_{\min};\cA) = \lim_{z\nearrow E_{\min}} \!\! N(z;\cA) \leq n(1/2; T(E_{\min}))<\infty. \qedhere
	\end{equation*}
\end{proof}

\begin{remark}
	Whenever $\min\sess(H)\notin\Sigma_2$, condition \eqref{cond:alpha,beta,gamma} is sharp in the sense that, if $\alpha^*+\gamma^* \geq 2\beta^*+d$, then the operator $H$ 
	may have an infinite number of eigenvalues below $\min\sess(H)$. This occurs, for example, for $d=1$, $\Omega=[-\pi,\pi]$, and  
	\begin{equation*}
	\begin{array}{l}
	w_1(x)=1+\sin^2(x), \quad v_1(x,y)=\sqrt{\dfrac{3}{\pi}} \sin(x), \\
	w_2(x,y)=\varepsilon(x)+2\varepsilon(x+y)+\varepsilon(y),
	\end{array}
	\quad x,y\in [-\pi,\pi],
	\end{equation*}
	where $\varepsilon(x):=1-\cos(x)$, $x\in[-\pi,\pi]$, see \cite{MNR-2016-1D} for more details.
\end{remark}
\section{Concluding remarks}
We conclude the paper with some remarks on modifications of our assumptions and results, and on the case $a=\infty$.

\medskip

\noindent
{\bf 6.1.} \ The uniqueness of the point $(x_0,y_0)\in\Omega$ in Assumption~\ref{ass:B} was assumed just for simplicity. Theorem~\ref{thm:fin.disc.spec} can be generalized if we assume 
	that there exist finitely many points $(t_j,t_j)$, $j=0,1,\ldots,N$, and constants $\alpha_j\geq0$, $\beta_j\in\RR$, $\gamma_j\geq0$, $j=0,\ldots,N$, $C_1, C_2, C_3>0$ and 
	$\delta>0$ with 
	\[
	\delta< \min_{\substack{k\neq l\\0\leq k,l\leq N}} \frac{1}{2} \dist((t_k,t_k),(t_l,t_l))
	\]
	such that for a.a.\ $x,y\in\Omega$ and each $j=0,1,\ldots,N$,
	\begin{enumerate}
		\smallskip
		\item  $w_2(x,y)-\min\sess(H) \geq C_1 \prod_{j=0}^N \Phi_{\alpha_j}(x-t_j, y-t_j)$,
		\smallskip
		\item $\chi_{B_{\delta}(t_j)}(y) |v_1(x,y)|\leq C_2\Phi_{\beta_j}(0,y-t_j)$,
		\smallskip
		\item $\Delta(x; \min\sess(H))\geq C_3\Phi_{\gamma_j}(x-t_j,0)$
		\smallskip
	\end{enumerate}
	with $\Phi_s$ as in \eqref{Phi_s}. Defining the constants $\alpha_j^*$, $\beta_j^*$ and $\gamma_j^*$ in an analogous way and replacing condition \eqref{cond:alpha,beta,gamma} 
	by 
	\[
	\alpha_j^*+\gamma_j^*<2\beta_j^*+d, \quad j=0,1,\ldots,N,
	\]
	the same analysis as above in a sufficiently small neighborhood of every point $(t_j,t_j)$ shows that the number of eigenvalues below $\min\sess(H)$ remains finite in this case. 
	
	Note that this is no longer true in general if the number of such points is infinite, see \cite{MNR-2016-1D} for an example in the smooth setting.
	
\medskip	

\noindent
{\bf 6.2.} \ We mention that in the previous studies, e.g.\ in \cite{MNR-2016-1D}, it was always assumed that the parameter function $w_2$ has a unique non-degenerate global 
	minimum, which implies that $\alpha^*=2$ in our Assumption~\ref{ass:B}. 
	
	While uniqueness in \cite{MNR-2016-1D} was assumed just for simplicity, our analysis shows that the non-degeneracy of the global minimum in \cite{MNR-2016-1D} is not always needed 
	to guarantee the finiteness of the discrete spectrum below $\min\sess(H)$, at least if $\min\sess(H)\notin\Sigma_2$.
	
\medskip

\noindent
{\bf 6.3.} \ Under assumptions analogous to Assumptions~\ref{ass:B}, \ref{ass:C} and with the same method, 
	one immediately obtains an analogue of Theorem~\ref{thm:fin.disc.spec} guaranteeing the finiteness of discrete eigenvalues above $\max\sess(H)$.
	
\medskip

\noindent
{\bf 6.4.} \ Motivated by the application to spin-boson Hamiltonians on the torus in $\RR^d$
(which was studied e.g.\ in \cite{MNR-2016-1D}), we focused on the case $a<\infty$ throughout the paper. However, our methods are of local nature and thus readily apply to the case of $\Omega=\RR^d$ where $a=\infty$. By requiring $v_1$ to have a compact support in $\RR^2$ and the conditions \eqref{assA:ess.sup.cond} of Assumption~\ref{ass:A} to hold on the support of $v_1$, and 
assuming the rest of the hypotheses in Assumption~\ref{ass:A} as well as in Assumptions~\ref{ass:B}, \ref{ass:C} for $\Omega=\RR^d$, we obtain the same conclusions of Theorems~\ref{thm:ess.spec.1} and \ref{thm:fin.disc.spec}. 

{\small
	\bibliographystyle{acm}
	\bibliography{spin_boson}
}
\end{document}